\documentclass[smallextended,envcountsect,]{svjour3}
\smartqed
\usepackage{graphicx}

\usepackage{latexsym, amsmath,amsxtra, amssymb, latexsym, amscd}
\usepackage{enumerate,cite,color}
\usepackage{hyperref}
\usepackage{comment}

\def\Limsup{\mathop{{\rm Lim}\,{\rm sup}}}

\def\ox{\bar{x}}

\def\Limsup{\mathop{{\rm Lim}\,{\rm sup}}}

\def\gph{\mbox{\rm gph}\,}
\def\epi{\mbox{\rm epi}\,}

\def\R{\mathbb{R}}
\def\N{\mathbb{N}}

\def\gph{\mathrm{gph}\,}
\def\epi{\mathrm{epi}\,}

\begin{document}

\title{Second-Order Optimality Conditions for  Sparse Differentiable Optimization Problems via Limiting Second-Order Subdifferentials}

\titlerunning{Second-order optimality conditions for  sparse differentiable optimization}

\authorrunning{L.T.T. Huyen, L. Jiao, N.V. Tuyen}


\author{Luu Thi Thu Huyen \and Liguo Jiao  \and Nguyen Van Tuyen}

\institute{
           Luu Thi Thu Huyen \at
           Faculty of Natural Sciences, Hung Vuong University,
           \\
            Phu Tho, Vietnam;
           \\
           Department of Mathematics, Hanoi Pedagogical University 2,\\
           Xuan Hoa, Phu Tho, Vietnam \\
           luuthithuhuyen@hvu.edu.vn
           \and Liguo Jiao \at Academy for Advanced Interdisciplinary Studies, Northeast Normal University,
           \\
           Changchun 130024, Jilin Province, People's Republic of China
           \\
           jiaolg356@nenu.edu.cn
             \and
           Nguyen Van Tuyen,  Corresponding author \at
           Department of Mathematics, Hanoi Pedagogical University 2,\\
           Xuan Hoa, Phu Tho, Vietnam \\
           nguyenvantuyen83@hpu2.edu.vn; tuyensp2@yahoo.com                         
}

\date{Received: date / Accepted: date}

\maketitle

\begin{abstract}
In this paper, we investigate a class of sparse optimization problems in which both the objective and constraint functions are Fr\'echet differentiable and possess locally Lipschitz continuous gradient mappings. More precisely, by utilizing the limiting (Mordukhovich) second-order subdifferential of the associated Lagrangian function, we establish new second-order necessary and sufficient optimality conditions for local optimal solutions. The obtained results are derived under mild assumptions and extend several existing results in the literature. In addition, we apply our theoretical developments to sparse multiobjective optimization problems and derive second-order sufficient optimality conditions for efficient solutions. Several examples are also presented to demonstrate the applicability and effectiveness of the proposed results.
\end{abstract}
\keywords{Sparse optimization \and second-order optimality conditions \and limiting second-order subdifferential \and constraint qualifications \and  sparse multiobjective optimization}
\subclass{49K30 \and 90C30 \and 90C33 \and  49J52 \and 49J53 \and 90C29}


\section{Introduction}

Sparse optimization has attracted significant attention in recent years due to its wide range of applications in signal processing, machine learning, statistics, compressed sensing, image reconstruction, portfolio optimization, and control theory. A standard sparse optimization problem (SP) can be formulated as
\begin{equation*}
	\min f(x) \ \ \text{s.t.}\ \ x\in \Phi, \|x\|_0\leq s,
\end{equation*}
where $f\colon\R^n\to \R$ is a given objective function, $\Phi\subset \R^n$ is a constraint set, $s\in (0, n)$ is a given positive integer, and $\|\cdot\|_0$ denotes the number of nonzero components of a vector. The cardinality constraint $\|x\|_0\leq s$  is a fundamental tool for enforcing sparsity of solutions.

\medskip

The growing interest in sparse optimization is mainly motivated by its practical importance. In compressed sensing, sparse optimization is used to recover signals from incomplete measurements; see, e.g., \cite{CandesRombergTao2006,Donoho2006}. In statistics and machine learning, sparse models play a central role in variable selection and high-dimensional regression; see, e.g.,   \cite{Tibshirani1996,Hastie2009}. Sparse portfolio selection has also been extensively studied in mathematical finance, while sparse control problems arise naturally in networked systems and optimal control; see, e.g.,  \cite{Bienstock1996,Tillmann2024}. Since sparse solutions are often more interpretable, stable, and computationally efficient, sparse optimization has become one of the most active topics in modern optimization theory.

\medskip

From the mathematical viewpoint, sparse optimization problems are highly challenging because the cardinality function is discontinuous and nonconvex. Even when the objective function and the remaining constraints are smooth, the feasible set generated by the cardinality constraint is nonconvex and lacks a regular geometric structure. Consequently, many classical techniques from nonlinear programming cannot be directly applied. In particular, the derivation of optimality conditions and constraint qualifications becomes substantially more delicate than in standard smooth optimization problems.

\medskip

To overcome these difficulties, several approaches have been proposed in the literature. One important direction reformulates sparse optimization problems as mathematical programs with complementarity constraints  or disjunctive programs; see, e.g.,  \cite{Benko-et-al-22,BucherSchwartz2018,BurdakovKanzowSchwartz2016,Gfrerer2014,KanzowSchwartz2013,LiangYe2021,Mehlitz,Xiao-Ye-24,XuYe}. Another influential framework is based on variational analysis and generalized differentiation; see the monographs of Mordukhovich \cite{Mordukhovich_2006,Mordukhovich_2018} and Rockafellar and Wets \cite{Rockafellar1998}. These techniques provide powerful tools for analyzing nonsmooth and nonconvex structures arising in sparse optimization.

\medskip

First- and second-order optimality conditions for sparse optimization have been intensively investigated over the last decade. For instance, Beck and Elda  \cite{BeckEldar2013} introduced and investigated three types of first-order necessary optimality conditions for the \eqref{Problem} problem with a single sparsity constraint, namely, basic feasibility, L-stationarity, and coordinate-wise optimality. Kanzow and Schwartz \cite{KanzowSchwartz2013} studied continuous reformulations of sparse optimization problems and derived several stationarity conditions. Subsequently, Beck and Hall \cite{BeckHalla} generalized these results to optimization problems involving a continuously differentiable objective function over sparse symmetric sets, which constitute a particular class of feasible sets for the \eqref{Problem} problem. Recently, Lu \cite{Lu-2015} studied the notion of strong L-stationarity for optimization problems over sparse symmetric sets considered in \cite{BeckHalla}. In addition, for the \eqref{Problem} problem, Lu and Zhang \cite{LuZhang} established first-order necessary optimality conditions by employing subspace techniques under the Robinson constraint qualification. In \cite{Pan-Xiu-Zhou}, Pan et al. derived explicit formulas for the Bouligrand/Clarke tangent cone and the regular/Clarke normal cone to 
$$\mathcal{S}:=\{x\in\R^n\,:\, \|x\|_0\leq s\}$$ 
and established the corresponding first- and second-order optimality conditions. More recently, Pan et al. \cite{Pan-Luo-Xiu_2017} and Pan et al. \cite{Pan-Xiu-Fan} introduced the so-called restricted linear independence constraint qualification, restricted Robinson constraint qualification, and restricted Mangasarian--Fromovitz constraint qualification, and established first- and second-order optimality conditions for problem \eqref{Problem} with operator or functional constraints. Other constraint qualifications, together with first- and second-order optimality conditions for problem \eqref{Problem} with functional constraints, can be found in \cite{BurdakovKanzowSchwartz2016,Flegel-et-al-07,Henrion-Outrata-05,Movahedian-et-al-19,Movahedian-et-al,Xiao-Ye-24} and the references therein. We emphasize that all the aforementioned results concerning second-order optimality conditions for problem \eqref{Problem} were derived under rather strong assumptions, such as the $C^2$-smoothness of the objective and constraint functions.

\medskip

However, many practical optimization models involve functions that are only continuously differentiable rather than twice continuously differentiable. Such situations naturally arise in machine learning, composite optimization, and robust optimization models, for example, sparse Huber regression \cite{Pan-Wang-etal} and $\ell_2-\ell_p$ Tikhonov regularization \cite{Borges-etal}. In this framework, classical Hessian-based approaches are generally no longer applicable, and one must instead employ generalized second-order constructions such as second-order epi-derivatives \cite{Rockafellar1998,Borwein2005,Khanh-Tuan_2007,Mohammadi-etal} and second-order subdifferentials \cite{Mordukhovich_2006,Mordukhovich_2018,Mordukhovich_2024,An-Xu-Yen,An-Tuyen,Huy-Tuyen-16,Huy-Tuyen-19,Tuyen-Huy-Kim,ChieuLeeYen2017,Khanh-etal-25}. Therefore, the development of second-order optimality conditions for sparse optimization problems with non-$C^{2}$ data is of significant theoretical and practical importance.

\medskip

In \cite{ChenJW-et_al_2024}, Chen et al. established second-order optimality conditions for nonsmooth sparse multiobjective optimization problems by employing Dini directional derivatives of the objective functions together with the Bouligand tangent cone and the second-order tangent set of the sparse set. Kan and Song \cite{Kan-Song} derived second-order optimality conditions for the existence of augmented Lagrange multipliers in optimization problems with sparsity and abstract constraints via the second-order epi-derivative of the augmented Lagrangian. More recently, Chen et al. \cite{ChenJW-et-at-2026} investigated second-order optimality conditions for sparsity-constrained optimization problems under assumptions involving the Fr\'echet second-order subdifferential of the Lagrangian function, instead of the twice continuous differentiability of the objective and constraint functions. Nevertheless, the Fr\'echet second-order subdifferential is, in some sense, rather restrictive, since it may be empty even for $C^{1,1}$ functions; see Example \ref{Example-31} below. In such cases, the corresponding second-order necessary optimality conditions become meaningless.

\medskip

Motivated by the above observations, this paper is devoted to the study of second-order optimality conditions for sparse optimization problems with  $C^{1,1}$ data by employing the so-called limiting (Mordukhovich) second-order  subdifferential. It is worth noting that the limiting second-order  subdifferential of $C^{1,1}$ functions is always nonempty and compact. The results obtained in this work are new.

\medskip

The paper is organized as follows. In Section \ref{Section2}, we recall preliminary notions and basic results from variational analysis and generalized differentiation. Section \ref{Section 3} is devoted to the derivation of second-order necessary optimality conditions. In Section \ref{Section 4}, we establish second-order sufficient conditions for local optimality. Applications to multiobjective
optimization problems are presented in Section \ref{Section 5}. Finally, concluding remarks are presented in the last section.

\section{Preliminaries} \label{Section2}
Throughout this work we deal with the Euclidean space $\mathbb{R}^n$ equipped with the usual scalar product $\langle \cdot, \cdot \rangle$ and the corresponding norm $\| \cdot\|.$ For a nonempty set $X \subset \mathbb{R}^n,$ the closure,  convex hull and conic hull of $X$ are denoted, respectively, by $\mathrm{cl}\,X$,  $ \mathrm{co}\,X$, and $\mathrm{cone}\,X.$

Let  $e_i$, $i\in \{1, \ldots, n$\},  be the {\em  $i$-th unit vector} of $\R^n$. For a given subset $J\subset\{1, \ldots, n\}$, we denote by $\text{span}\,\{e_i\,:\, i\in J\}$ the subspace of $\R^n$ spanned by $\{e_i\,:\, i\in J\}$.  

For a set-valued mapping $F \colon \mathbb{R}^n \rightrightarrows \mathbb{R}^m$, the {\em Painlev\'e--Kuratowski outer limit} of $F$ at $\bar x\in\R^n$ is defined by
\begin{equation*}
	\Limsup_{x \to {\bar x}} F(x) := \{y \in \mathbb{R}^m \,:\, \exists x_k \to {\bar x}, \exists y_k \in F(x_k), y_k \to y\}.
\end{equation*}

\begin{definition}[{\normalfont see \cite{Aubin1990,Khan2015}}]\rm 
	Let $X\subset \R^n$ and $\bar x\in X$. 
	\begin{enumerate}[\rm(i)]
		\item The {\em  contingent cone} (or the {\em   Bouligand-Severi tangent cone}) $T_X(\bar x)$ to $X$ at $\bar x$ is defined by
		\begin{equation*}
			T_X(\bar x):=\big\{u\in\R^n\,:\, \exists t_k\to  0^+, \exists u_k\to u, \bar x+t_ku_k\in X \ \ \forall k\in\N\big\}.
		\end{equation*}
		
		\item  The {\em  Clarke tangent cone} $T^C_X(\bar x)$  to $X$ at $\bar x$ is given by
		\begin{equation*}
			T^C_X(\bar x):=\big\{u\in\R^n:\forall x_k\xrightarrow{X}\bar x,  \forall t_k\to  0^+, \exists u_k\to u,  x_k+t_ku_k\in X \ \ \forall k\in\N\big\}.
		\end{equation*}
		\item  The {\em  radial tangent cone} $\mathcal{R}_X(\bar x)$  to $X$ at $\bar x$ is defined by
		\begin{equation*}
			\mathcal{R}_X(\bar x):=\big\{u\in\R^n \,:\, \exists\tau>0,   \bar x+t u\in X \ \ \forall  t\in [0, \tau]\big\}.
		\end{equation*}
	\end{enumerate}	
\end{definition}

The following remark summarize some properties of tangent cones that can be found in  \cite{Bonnans-Shapiro,Borwein2005,Khan2015,Rockafellar1998}.

\begin{remark}\rm 
	\begin{enumerate}[(i)]
		\item The contingent and Clarke tangent cones are closed, whereas it can happen that the radial cone is not closed; see Example \ref{Example-tangent}(ii) below. The Clarke tangent cone is always convex. 
		
		\item For $X_1, \ldots, X_p \subset\R^n$ and $\bar x\in X:=\bigcap_{i=1}^pX_i$, it holds that
		\begin{equation*}
			T_X(\bar x) \subset \bigcap_{i=1}^p T_{X_i}(\bar x).
		\end{equation*}  
		
		\item The following inclusions hold:
		\begin{equation*}
			\mathcal{R}_X(\bar x) \subset T_X(\bar x), \  \   T^C_X(\bar x)\subset T_X(\bar x),
		\end{equation*}
		and when $X$ is convex, then $\mathcal{R}_X(\bar x)=\mathrm{cone}(X-\bar x)$ and
		\begin{equation*}
			T^C_X(\bar x)= T_X(\bar x)=\mathrm{cl}\,\mathcal{R}_X(\bar x).
		\end{equation*}

	\end{enumerate}
\end{remark}

The next example shows  that there is no inclusion relation between $\mathcal{R}_X(\bar x)$ and $T^C_X(\bar x)$.
\begin{example}\label{Example-tangent}\rm 
	\begin{enumerate}[(i)]
		\item Let $X=(\R_+\times\{0\})\cup \{0\}\times\R_+$ and $\bar x=(0,0)$. Then, it is easy to check that $\mathcal{R}_X(\bar x)=X$ and $T^C_X(\bar x)=\{(0,0)\}$ and so $$\mathcal{R}_X(\bar x)\nsubseteq T^C_X(\bar x).\footnote{This shows that Remark 2.1(i) in \cite{ChenJW-et-at-2026} is not correct.}$$
		
		\item Let $X=\{x=(x_1, x_2)\in\R^2\,:\, x^2_1\leq x_2\}$ and $\bar x=(0,0)$. Then $X$ is convex, 
		\begin{equation*}
			\mathcal{R}_X(\bar x)=\mathrm{cone}(X-\bar x)=\{(x_1, x_2)\in\R^2\,:\, x_2>0\}\cup\{(0,0)\},
		\end{equation*}	
		and
		\begin{equation*}
			T^C_X(\bar x)=T_X(\bar x)=\mathrm{cl}\,\mathcal{R}_X(\bar x)=\{(x_1, x_2)\in\R^2\,:\, x_2\geq 0\}.\footnote{This shows that \cite[Lemma 3.2]{ChenJW-et-at-2026}, and consequently \cite[Theorem 3.3(ii)]{ChenJW-et-at-2026}, appear to be incorrect.}
		\end{equation*}
		Hence $T^C_X(\bar x)\nsubseteq \mathcal{R}_X(\bar x)$.	
	\end{enumerate}	
	
\end{example} 

\begin{definition} [{\normalfont see \cite[pp. 5--6]{Mordukhovich_2018}}] 
	\rm	Let $X$ be a nonempty subset of $\mathbb{R}^n$ and $\bar x \in X$.
	\begin{enumerate}[\rm(i)]
		\item The \textit{Fr\'echet (regular) normal cone}  to $X$ at $\bar x$ is defined by
		\begin{align*}
			\widehat N_X(\bar x)=\Big\{ v\in \mathbb{R}^n : \limsup\limits_{x \xrightarrow{X}\bar x} \dfrac{\langle v, x-\bar x \rangle}{\|x-\bar x\|} \leq 0 \Big\},
		\end{align*}
		where $x \xrightarrow{X} \bar x$ means that $x \rightarrow \bar x$ and $ x\in X$.
		\item The \textit{limiting (Mordukhovich) normal cone} to $X$ at $\bar x$ is given by
		\begin{align*}
			N_X(\bar x)=\Limsup\limits_{ x \xrightarrow{X} \bar x} \widehat{N}_X(x).
		\end{align*}
		\item The \textit{Clarke normal cone} to $X$ at $\bar x$ is defined by $N^C_X(\bar x):=\mathrm{cl}\,\mathrm{co}\ N_X(\bar x)$.
	\end{enumerate}	
	
	We put $\widehat N_X(\bar x)=N_X(\bar x) =\emptyset$ if $\bar x \not\in X$.
\end{definition}	

By definition, it is clear that $\widehat N_X(\bar x) \subset N_X(\bar x)$  for all $\bar x \in X.$ If $X$ is convex, then the Fr\'echet normal cone, the limiting normal cone, and the normal cone in the sense of convex analysis coincide, i.e., 
\begin{align*}
	\widehat N_X(\bar x) = N_X(\bar x):=\{ v\in \mathbb{R}^n : \langle v,x- \bar x \rangle \le 0,  \ \forall x\in X\}.
\end{align*}

Let  $F: \mathbb{R}^n\rightrightarrows{\mathbb{R}^m}$ be a  multifunction with the \textit{domain} $${\rm{dom}}\, F:=\{ x \in \mathbb{R}^n : F(x)\not=\emptyset\}$$ and the \textit{graph} $${\rm{gph}}\, F:=\{ (x,y) \in \mathbb{R}^n \times \mathbb{R}^m : y \in F(x)\}.$$  

\begin{definition} [{\normalfont see~\cite[Definition 1.11]{Mordukhovich_2018}}]\rm  Let $  F: \mathbb{R}^n\rightrightarrows{\mathbb{R}^m}$ be a multifunction and  $(\bar x, \bar y) \in  {\rm{gph}}\, F$.	The  \textit{limiting (Mordukhovich) coderivative}  of  $F$ at $(\bar x, \bar y)$ is a multifunction $D^* F(\bar x, \bar y): \mathbb{R}^m \rightrightarrows{\mathbb{R}^n}$ with the values
	\begin{align*}
		D^* F(\bar x,\bar y)(v):=&\left\{u\in \mathbb{R}^n : (u, -v) \in  N_{\gph\, F}(\bar x, \bar y)\right\}, \quad v \in \mathbb{R}^m.
	\end{align*}
	If $(\bar x, \bar y) \notin {\rm{gph}}\, F$, one puts $ D^* F(\bar x, \bar y)(v)=\emptyset$ for any $v\in \mathbb{R}^m$. The symbol $D^* F(\bar x)$ is used when $F$ is single-valued at $\bar x$ and $\bar y=F(\bar x)$. 
\end{definition}

Consider a function $\varphi: \mathbb{R}^n\rightarrow \overline{\mathbb{R}}:= \mathbb{R} \cup \{+\infty\}$ with the \textit{effective domain} $$\mbox{dom}\, \varphi:=\{x \in \mathbb{R}^n : \varphi(x) < +\infty\},$$ the \textit{epigraph} $$ {\rm{epi}}\, \varphi:=\{ (x, \alpha) \in \mathbb{R}^n \times \mathbb{R} : \alpha \ge \varphi (x)\},$$
and the \textit{hypergraph} $$ {\rm{hypo}}\, \varphi:=\{ (x, \alpha) \in \mathbb{R}^n \times \mathbb{R} : \alpha \le \varphi (x)\}.$$  

\begin{definition}[{\normalfont see~\cite[Definition 1.77]{Mordukhovich_2006}}] \rm  
	Let $\bar {x} \in \mbox{dom}\, \varphi$. The  \textit{limiting (Mordukhovich) subdifferential} of $\varphi$ at $\bar {x}$ is defined by
	\begin{align*}
		\partial \varphi(\bar x):=\left\{v \in \mathbb{R}^n : (v, -1) \in  N_{\epi \varphi}(\bar {x}, \varphi(\bar{ x}))\right\}.
	\end{align*}
	If $\bar {x} \notin \mbox{dom}\, \varphi$, then we put  $\partial \varphi (\bar x):=\emptyset$.
\end{definition}

One can use the notion of coderivative to construct the second-order generalized differential theory of extended-real-valued functions.

\begin{definition}  [{\normalfont see \cite[Definition 3.17]{Mordukhovich_2018}}] {\rm Let $\varphi:\mathbb{R}^n\rightarrow \overline{\mathbb{R}}$ 	be a function and  $\bar{x}\in\mathrm{dom}\,\varphi.$ For any $\bar y\in  \partial \varphi (\bar x)$, the map $\partial^2 \varphi(\bar x,\bar y):\mathbb{R}^n\rightrightarrows \mathbb{R}^n$ with the values
		\begin{align*}
			\partial^2 \varphi(\bar x,\bar y)(v):=(D^* \partial \varphi)(\bar x,\bar y)(v)=\{u : (u,-v)\in N_{\gph \partial \varphi}(\bar x, \bar y)\}
		\end{align*} is said to be the \textit{limiting (Mordukhovich)
			second-order subdifferential} of $\varphi$ at $\bar x$ relative to	$\bar y.$}
\end{definition}

If $\partial \varphi (\bar x)$ is a singleton, the symbol $\bar y$ in the notation $\partial^2 \varphi (\bar x,\bar y)(v)$ will be omitted.  When $\varphi$ is a $C^2$-smooth function around $\bar x$, i.e., $\varphi$ is twice continuously differentiable in a neighborhood of $\ox$, then  
\begin{align*}
	\partial^2 \varphi(\bar x)(v)=\{\nabla^2 \varphi (\ox)^* v\}= \{\nabla^2 \varphi(\bar x)v\} \ \  v\in \mathbb{R}^n
\end{align*}
with $\nabla^2 \varphi (\ox)$ being the Hessian matrix of $\varphi$ at $\ox$; see, e.g., \cite[p. 124]{Mordukhovich_2018}.

\medskip
Let $U\subset\R^n$ be a nonempty, open, and convex set. We denote by  $C^{1,1}(U)$ the class of all real-valued functions $\varphi$, which are Fr\'echet differentiable on $U$, and whose gradient mapping $\nabla \varphi (\cdot)$ is locally Lipschitz on $U$. By \cite[Theorem 1.90]{Mordukhovich_2006}, if $\varphi\in C^{1,1}(U)$ and $\bar x\in U$, then one has
\begin{align*} 
	\partial^2 \varphi(\bar x) (v):= \partial^2 \varphi (\bar x, \nabla \varphi (\bar x))(v)=\partial \langle v, \nabla \varphi \rangle (\bar x)\ \ \forall v\in \mathbb{R}^n.
\end{align*}

We say that the function $\varphi$ is of {\em class $C^{1,1}$ around $\bar x$} if there exist an open neighborhood $U$ of $\bar x$ such that $\varphi\in C^{1,1}(U)$.

\medskip

The following properties can be obtained directly from the definition.
\begin{proposition}\label{property}
	Let $\varphi$ be of class $C^{1,1}(U)$ and $\bar x\in U$. The following assertions hold:
	\begin{enumerate}[\rm (i)]
		\item $\partial^2 \varphi(\bar x)(\lambda v) = \lambda \partial^2 \varphi(\bar x)(v)$ for all $\lambda \ge 0$ and $v\in \mathbb{R}^n.$
		\item For any $v\in \mathbb{R}^n$ the mapping $x \mapsto \partial^2 \varphi(x) (v)$ is locally bounded. Moreover, if $x_k \to \bar x$, $x_k^* \to x^*$, $x_k^* \in \partial^2 \varphi(x_k)(v)$ for all $k\in \mathbb{N}$, then $x^*\in \partial^2 \varphi(\bar x)(v).$
	\end{enumerate} 
\end{proposition}

\medskip

The Taylor formula for $C^{1,1}$ functions involving the limiting second-order subdifferential plays a crucial role in our study.
\begin{theorem} [{\normalfont see~\cite[Theorem 3.1]{Feng_Li_2020}}]\label{Taylor_formula}
	Let $\varphi$ be of class $C^{1,1}(U)$  and $a, b\in U$. Then, there exist $z\in \partial^2 \varphi(x) (b-a)$ and $z'\in \partial^2 \varphi(x') (b-a)$ for some $x$ and $x' \in [a,b]$  such~that
	$$ \dfrac{1}{2} \langle z', b-a \rangle \le \varphi(b)-\varphi(a)-\langle\nabla \varphi(a), b-a \rangle \le \dfrac{1}{2}\langle z, b-a \rangle.$$
\end{theorem}

\section{Second-Order Necessary Optimality Conditions}\label{Section 3} 
Consider the following sparse optimization problem
\begin{equation}\label{Problem}
	\min f(x)\ \ \text{s.t.}\ \ x\in\Omega:=\{x\in\mathbb{R}^n\,:\,\phi(x)\in D,   \|x\|_0\leq s\},\tag{SP}	
\end{equation}
where $f\colon\R^n\to \R$, $\phi_j\colon\R^n\to\R$, $j=1, \ldots, m$,  are of class $C^1$ functions, $D\subset\R^m$ is a  closed and convex cone, and $s\in (0,n)$ is any given positive
integer. Let us denote $\Phi$  and $\mathcal{S}$, respectively, by
\begin{equation*}
	\Phi:=\phi^{-1}(D) \ \ \text{and}\ \ \mathcal{S}:=\{x\in\R^n\,:\, \|x\|_0\leq s\}. 
\end{equation*}
Then $\Omega=\Phi\cap   \mathcal{S}$ and the sparse constraint set $\mathcal{S}$ can be formulated as
\begin{equation*}
	\mathcal{S}=\bigcup_{J\in\mathcal{J}} \R^n_J,
\end{equation*}
where $\mathcal{J}:=\{J\subset \{1, \ldots, n\}\,:\, |J|=s\}$ and $$\R^n_J:=\mathrm{span}\{e_i\,:\, i\in J\}.$$

For $x\in \mathcal{S}$, the support index set $\Gamma(x)$ of $x$ is defined by
\begin{equation*}
	\Gamma(x):=\{i\in \{1, \ldots, n\}\,:\, x_i\neq 0\}.
\end{equation*} 
For a subset $I$ of $\{1, \ldots, n\}$ and $x\in\R^n$, we denote $x_I:=(x_i)_{i\in I}$ and the cardinality of $I$ by $|I|$. The complement  of $I$ is denoted by $I^c$.

\medskip

For $\bar x\in \Omega$, we denote {\em  hereafter} $\bar J:=\Gamma(\bar x)$ and $\mathcal{J}(\bar x):=\{J\in\mathcal{J}\,:\, J\supset \bar J\}$.

\begin{definition}[{\normalfont see \cite[Definition 3.1]{Pan-Luo-Xiu_2017}}] \rm 
	Let $0\neq \bar x\in \Omega$  and $\bar \lambda\in N_D(\phi(\bar x))$. We say that:
	\begin{enumerate}[\rm(i)]
		\item  The {\em  Robinson's constraint qualification} (RCQ) for $\Phi$ holds at $\bar x$ if 
		\begin{equation*}
			0\in\mathrm{int}\,\{\phi(\bar x)+\nabla \phi(\bar x)\R^n-D\}.
		\end{equation*}
		\item The {\em  restricted Robinson's constraint qualification} (RRCQ) holds at $\bar x$ if
		\begin{equation*}
			0\in\mathrm{int}\,\{\phi(\bar x)+\nabla \phi(\bar x)\R^n_{\bar J}-D\}.
		\end{equation*}
		\item The  {\em  restricted strict Robinson's constraint qualification} (RSRCQ) holds at the pair $(\bar x, \bar\lambda)$ if
		\begin{equation*}
			0\in\mathrm{int}\,\{\phi(\bar x)+\nabla \phi(\bar x)\R^n_{\bar J}-D_0\},
		\end{equation*}
		where $D_0:=\{y\in D\,:\, \langle\bar\lambda, y-\phi(\bar x)\rangle=0\}$.
	\end{enumerate}	
\end{definition}

\begin{remark}\rm 
	\begin{enumerate}[\rm(i)]
		\item The following implications hold:
		\begin{equation*}
			\text{RSRCQ} \Rightarrow \text{RRCQ} \Rightarrow \text{RCQ}.
		\end{equation*} 
		\item  (see \cite[Corollary 2.91]{Bonnans-Shapiro}) If the RCQ holds for $\Phi$ at $\bar x$, then
		\begin{equation*}
			T_\Phi(\bar x)=\{d\in\R^n\,:\, \nabla \phi(\bar x)d\in T_D(\phi(\bar x))\}.
		\end{equation*}
		
		\item (see \cite[p. 150]{Bonnans-Shapiro}) Let $\bar x\in \Phi$. Then, one has 
		$$N_D(\phi(\bar x))=\{\lambda\in\R^m\,:\, \lambda\in D^*, \langle\lambda, \phi(\bar x)\rangle=0\},$$
		where $D^*:=\{y^*\in\R^m\,:\, \langle y^*, y\rangle\leq 0 \ \ \forall y\in D\}$ is the dual cone of $D$.
	\end{enumerate}	
	
\end{remark}

\begin{lemma}[{\normalfont see \cite[Proposition 3.1]{Pan-Luo-Xiu_2017}}]\label{Inter-tangent} Let $\bar x\in \Omega$ be such that $\bar x\neq 0$. If the RRCQ holds at $\bar x$, then
	\begin{equation*}
		T_\Omega(\bar x)= T_\Phi(\bar x)\cap T_\mathcal{S}(\bar x)=(\nabla \phi(\bar x))^{-1}(T_D(\phi(\bar x)))\cap T_\mathcal{S}(\bar x).
	\end{equation*}
\end{lemma}

\begin{lemma}[{\normalfont see \cite[Theorems 2.1--2.2]{Pan-Xiu-Zhou}}]\label{Lemma-tan-nor} For any $\bar x\in\mathcal{S}$, we have:
	\begin{enumerate}[\rm(i)]
		\item $T^C_\mathcal{S}(\bar x)=\R^n_{\bar J}$ and  $N^C_\mathcal{S}(\bar x)=\R^n_{{\bar J}^c}$;
		\item $T_\mathcal{S}(\bar x)=\bigcup\{\R^n_J\,:\, J\supseteq \bar J, |J|= s\}$,  
		\begin{equation*}
			\widehat{N}_\mathcal{S}(\bar x)=
			\begin{cases}
				\R^n_{{\bar J}^c}, \ \ &\text{if}\ \ |\bar J|=s,
				\\
				\{0_{\R^n}\}, \ \ &\text{if}\ \ |\bar J|<s, 
			\end{cases}
		\end{equation*}
		and
		\begin{equation*}
			{N}_\mathcal{S}(\bar x)=
			\begin{cases}
				\R^n_{{\bar J}^c}, \ \ &\text{if}\ \ |\bar J|=s,
				\\
				\bigcup\{\R^n_{J^c}\,:\, J\supseteq \bar J, |J|= s\}, \ \ &\text{if}\ \ |\bar J|<s. 
			\end{cases}
		\end{equation*}
		
	\end{enumerate} 
	
\end{lemma}

The Lagrangian associated with  problem \eqref{Problem} is defined by
\begin{equation*}
	L(x, \lambda):=f(x)+\langle \lambda, \phi(x)\rangle \ \ \forall (x, \lambda)\in\R^n\times\R^m.
\end{equation*}
\begin{definition}[{\normalfont see \cite{Pan-Luo-Xiu_2017}}]\rm  Let $\bar x\in\Omega$. The sets of Lagrange multipliers in the sense of Fr\'echet and Clarke are defined, respectively, by
	\begin{align*}
		\Lambda^S(\bar x)&:=\{\lambda\in \R^m_+\,:\, -\nabla L_x(\bar x, \lambda)\in \widehat{N}_\mathcal{S}(\bar x), \lambda\in N_D(\phi(\bar x))\},
		\\
		\Lambda^C(\bar x)&:=\{\lambda\in \R^m_+\,:\, -\nabla L_x(\bar x, \lambda)\in {N}^C_\mathcal{S}(\bar x), \lambda\in N_D(\phi(\bar x))\}.
	\end{align*} 
\end{definition}

\begin{definition}\rm    Let $\bar x\in\Omega$. We say that:
	\begin{enumerate}[\rm(i)]
		\item $\bar x$ is a {\em  local optimal solution} to problem \eqref{Problem} if there exists a neighborhood $U$ of $\bar x$ such that $f(x)\geq f(\bar x)$ for all $x\in\Omega\cap U$. 
		
		\item (see \cite[Definition 2.1]{Auslender_1984}) $\bar x$ is an \textit{isolated local solution of order~2}  to problem \eqref{Problem} if there exist a  constant $c>0$ and a neighborhood $U$ of $\bar x$ such that
		\begin{equation*} 
			f (x)\geq  f (\bar x)  +\frac{1}{2}\, c\, \|x-\bar x\|^2\ \ \forall x\in \Omega \cap U.
		\end{equation*}
	\end{enumerate}	
	
\end{definition}

\begin{lemma}[{\normalfont  see \cite[Theorem 4.1]{Pan-Luo-Xiu_2017}}]\label{Lemma-Lagrange} Let $\bar x\neq 0$ be a local optimal solution to problem \eqref{Problem} and $\bar\lambda\in N_D(\phi(\bar x))$. 
	\begin{enumerate}[\rm(i)]
		\item If the RRCQ holds at $\bar x$, then $\Lambda^C(\bar x)$ is a nonempty, convex and compact set.
		\item If the RSRCQ holds at the pair $(\bar x, \bar\lambda)$, then $\Lambda^S(\bar x)$ is a singleton.
	\end{enumerate}
\end{lemma}

\subsection{Sparse Optimization  with Equality Constraints}
We first consider second-order optimality conditions for problem \eqref{Problem} with $D=\{0_{\R^m}\}$.
\begin{theorem}\label{Necessary-equality} Suppose that $\bar x\neq 0$ is a local optimal solution to \eqref{Problem} with $D=\{0_{\R^m}\}$,  and $f, \phi$ are of class $C^{1,1}$ around $\bar x$.
	\begin{enumerate} [\rm(i)]
		\item If the RRCQ holds at $\bar x$ and $\|\bar x\|_0=s$, then for each $u\in T_\Omega(\bar x)$ and $\bar \lambda\in \Lambda^C(\bar x)$, there exists   $z\in \partial^2 L(\bar x, \bar\lambda)(u)$ such that $\langle z, u\rangle\geq 0$.

		\item  If the RSRCQ holds at $\bar x$ and  $\bar \lambda$ is the unique Lagrange multiplier in $\Lambda^S(\bar x)$, then for each $u\in T_\Omega(\bar x)$, there exists   $z\in \partial^2 L(\bar x, \bar\lambda)(u)$ such that $\langle z, u\rangle\geq 0$. 
		
		\item  If $\bar x$ is an isolated local solution of order $2$ to \eqref{Problem} and the RSRCQ holds at $\bar x$ with respect to the unique Lagrange multiplier   $\bar \lambda\in\Lambda^S(\bar x)$, then for each $u\in T_\Omega(\bar x)\setminus\{0\}$, there exists   $z\in \partial^2 L(\bar x, \bar\lambda)(u)$ such that $\langle z, u\rangle> 0$. 
	\end{enumerate}	
	
\end{theorem}
\begin{proof} (i) Since the RRCQ holds at $\bar x$, $\Lambda^C(\bar x)$ is nonempty due to Lemma \ref{Lemma-Lagrange}.  Fix any $u\in T_\Omega(\bar x)$ and $\bar \lambda\in \Lambda^C(\bar x)$. Then $\bar \lambda\in N_D(\phi(\bar x))$ and by Lemma \ref{Lemma-tan-nor} one has
	\begin{equation*}
		-\nabla L_x(\bar x, \bar \lambda)\in {N}^C_\mathcal{S}(\bar x)=\R^n_{{\bar J}^c}. 	
	\end{equation*}	
	By the definition of the contingent cone, there exist sequence $t_k\to 0^+$ and $u_k\to u$  such that $x_k:=\bar x+t_ku_k\in \Omega$ for all $k\in\N$. Hence $x_k\in\mathcal{S}$ and $\phi(x_k)=\phi(\bar x)=0_{\R^m}$,  and so $\langle \bar \lambda, \phi(x_k)-\phi(\bar x)\rangle =0$ for all $k\in\N$. This implies that 
	\begin{equation}\label{New-equa-2}
		L(x_k, \bar \lambda)-L(\bar x, \bar\lambda)=f(x_k)-f(\bar x)\geq 0 
	\end{equation}
	for all $k$ large enough. Since the RRCQ holds at $\bar x$, $T_\Omega(\bar x)=T_\Phi(\bar x)\cap T_\mathcal{S}(\bar x)$. Furthermore, it follows from the fact that $\|\bar x\|_0=s$ and Lemma \ref{Inter-tangent} that $T_\mathcal{S}(\bar x)=\R^n_{\bar J}$. Hence, $u\in \R^n_{\bar J}$ and so $\langle \nabla L_x(\bar x, \bar \lambda), u\rangle =0$. Clearly $\bar x_k\to \bar x$ as $k\to\infty$. Thus 
	$$\bar J=\Gamma(\bar x)\subset \Gamma(x_k)$$ 
	for all $k$ larger enough. This and the facts that $\|\bar x\|_0=|\bar J|=s$ and $|\Gamma(x_k)|\leq s$ imply that $\Gamma(x_k)=\Gamma(\bar x)$ and so $x_k\in\R^n_{\bar J}$ for all $k$ large enough. Since $\R^n_{\bar J}$ is a linear subspace, $u_k=\frac{1}{t_k}(x_k-\bar x)\in \R^n_{\bar J}$ and we therefore get  $\langle \nabla L_x(\bar x, \bar \lambda), u_k\rangle =0$ for all $k$ large enough. 
	
	We now rewrite $L(x_k, \bar \lambda)-L(\bar x, \bar\lambda)$ as follows:
	\begin{align*}
		L(x_k, \bar \lambda)-L(\bar x, \bar\lambda) =[L(x_k, \bar \lambda)&-L(\bar x+t_ku, \bar \lambda)] 	
		\\
		&+[L(\bar x+t_ku, \bar \lambda) - L(\bar x, \bar\lambda)-t_k\langle \nabla L_x(\bar x, \bar \lambda), u\rangle].
	\end{align*}
Put $L_{1k}:=L(x_k, \bar \lambda)-L(\bar x+t_ku, \bar \lambda)$ and
	$$L_{2k}:=L(\bar x+t_ku, \bar \lambda) - L(\bar x, \bar\lambda)-t_k\langle \nabla L_x(\bar x, \bar \lambda), u\rangle.$$
By the assumption, it is clear that $L(\cdot, \bar\lambda)$ is of class $C^{1,1}$ around $\bar x$. Let $\ell$ be the Lipschitz constant of $\nabla_x L(\cdot, \bar\lambda)$ around $\bar x$. 
	
	By the  mean value theorem (see, e.g., \cite[Theorem~5.10]{Rudin_PMA1976}), there exists $\xi_k\in (\bar x+t_k u, x_k)$ such that
	\begin{equation*}
		{L}_{1k} = \langle \nabla_x{L}(\xi_k,\bar\lambda), t_k(u_k-u) \rangle.
	\end{equation*}
	This and the Lipschitzness of $L(\cdot, \bar\lambda)$ imply that
	\begin{align*}\nonumber
		|{L}_{1k}|&= |\langle \nabla_x L(\xi_k,\bar\lambda), t_k(u_k-u) \rangle| \\ \nonumber
		&= |\langle \nabla_x {L}(\xi_k,\bar\lambda) -   \nabla_x {L}(\bar x,\bar\lambda), t_k(u_k -u) \rangle| \\ \nonumber
		& \le t_k \| \nabla_x  {L}(\xi_k,\bar\lambda)-   \nabla_x \mathcal{L}(\bar x,\bar\lambda)\| \cdot \|u_k-u\|
		\\
		&\le \ell t_k \|\xi_k-\bar x\|\cdot \|u_k-u\|
		\\
		&\leq \ell t^2_k (\|u_k-u\|+\|u\|)\|u_k-u\|,
	\end{align*}
	where the second equality follows from the facts that 
	$$\langle \nabla L_x(\bar x, \bar \lambda), u\rangle= \langle \nabla L_x(\bar x, \bar \lambda), u_k\rangle =0.$$  
	This, together with the fact that $u_k\to u$ , implies that $\frac{L_{1k}}{t^2_k} \to 0$ as $k\to\infty$.
	
	On the other hand, for each $k\in\N$, by applying the Taylor formula (Theorem~\ref{Taylor_formula}) for the function $L(\cdot, \bar\lambda)$ on $[\bar x, \bar x+t_k u]$,  we can find  $v_k\in \partial^2 {L} (\eta_k, \bar \lambda) (t_k u)$ for some $\eta_k\in [\bar x, \bar x+t_k u]$ such that 
	\begin{align*} 
		{L}_{2k} \le  \dfrac{1}{2} \langle v_k, t_k u\rangle.
	\end{align*}
	By Proposition~\ref{property}(i), there exists $z_k\in \partial^2  {L} (\eta_k, \bar \lambda)(u)$ such that $v_k=t_kz_k$. Hence
	\begin{align*} 
		{L}_{2k} \le  \dfrac{1}{2}t_k^2 \langle z_k, u\rangle.
	\end{align*}
	This,  together with \eqref{New-equa-2}, implies that
	\begin{align*}
		0\le {L}_{1k} + {L}_{2k} \le  \dfrac{1}{2}t_k^2 \langle z_k, u\rangle+ {L}_{1k},
	\end{align*}
	or, equivalently,
	\begin{equation}\label{New-equa-1}
		0\le  \dfrac{1}{2} \langle z_k, u\rangle+\dfrac{ {L}_{1k}}{t_k^2}.
	\end{equation}
	By the Lipschitzness of $ {L} (\cdot, \bar \lambda)$  around $\bar x$,  $\partial^2  L (\cdot, \bar \lambda)(d) $ is locally bounded around $\bar x$. This, together with the fact that  $\eta_k \to \bar x$, implies that $\{z_k\}$ is bounded. Hence, we may assume that $z_k$ converges to some $z\in\partial^2  {L} (\bar x, \bar \lambda)(u)$. By passing to the limit as $k\to \infty$ in \eqref{New-equa-1}, we obtain $$ \langle z, u \rangle  \ge 0,$$
	as required.
	
	\medskip
	
	(ii) Since the RSRCQ implies the RRCQ, it suffices, in proving this assertion, to consider the case where $\|\bar x\|_0<s$. For $u\in T_\Omega(\bar x)$, there exist sequences $t_k\to 0^+$ and $u_k\to u$ such that $x_k:=\bar x+t_k u_k\in \Omega$ for all $k\in \N$. By Lemma \ref{Lemma-tan-nor}(ii), one has $\widehat{N}_\mathcal{S}(\bar x)=\{0_{\R^n}\}$. Hence, for $\bar\lambda\in\Lambda^S(\bar x)$, $\nabla_x L(\bar x, \bar\lambda)=0$ and so $\langle \nabla L_x(\bar x, \bar \lambda), u\rangle= \langle \nabla L_x(\bar x, \bar \lambda), u_k\rangle =0$. An argument similar to that used in the proof of part (i) shows that there exists $z\in \partial^2 L(\bar x, \bar\lambda)(u)$ satisfying $\langle z, u\rangle\geq 0$. The proof is complete.   
	
	\medskip
	
	(iii) Since $\bar x$ is an isolated local solution of order $2$  to problem \eqref{Problem}, there exist a  constant $c>0$ and a neighborhood $U$ of $\bar x$ such that
	\begin{equation*} 
		f (x)\geq  f (\bar x)  +\frac{1}{2}\, c\, \|x-\bar x\|^2\ \ \forall x\in \Omega \cap U.
	\end{equation*} 
	For $u\in T_\Omega(\bar x)\setminus\{0\}$, there exist $t_k\to 0^+$ and $u_k\to u$ such that $x_k:=\bar x+t_k u_k\in \Omega \cap U$ for all $k\in\N$. Hence
	\begin{equation*}
		L(x_k, \bar \lambda)- L(\bar x, \bar\lambda)=f(x_k)-f(\bar x)\geq 	\frac{1}{2}\, c\, t_k^2\|u_k\|^2 \ \ \forall k\in\N.
	\end{equation*}
	By this and a similar argument as in the proof of part (ii), we can show that there exists $z\in\partial^2 L(\bar x, \bar\lambda)(u)$ such that
	\begin{equation*}
		\frac{1}{2}\langle z, u\rangle\geq \frac{1}{2}\, c\,\|u\|^2>0
	\end{equation*}
	as required. \qed
\end{proof}

\begin{remark}\rm If the RSRCQ does not hold at $\bar x$, then so is Theorem \ref{Necessary-equality}; see \cite[Example 3.1]{ChenJW-et-at-2026}. Furthermore, in Theorem \ref{Necessary-equality}, the claim {\em ``there exists $z\in\partial^2L(\bar x, \bar\lambda)(u)$''} cannot be replaced by  {\em``for all $z\in\partial^2L(\bar x, \bar\lambda)(u)$''}. To see this, let us consider the following example.   	
\end{remark}

\begin{example}\label{Example-31}\rm Let $n=3$, $s=1$, $m=1$,   $\phi(x)=x_1-1$,   $D=\{0\}$, and $$f(x)=\frac{1}{2}(x_2-1)|x_2-1|.$$
	Then 
	$$\Omega=\{x\in\R^3\,:\, x_1-1=0, \|x\|_0\leq 2\}$$
	and $\bar x=(1,1,0)\in\Omega$. Then
	\begin{equation*}
		\nabla f(x)=(0, |x_2-1|, 0) \ \ \forall x\in\R^3.
	\end{equation*}
	Clearly, $\bar x$ is a local optimal solution to \eqref{Problem}. We see that $\bar J=\Gamma(\bar x)=\{1, 2\}$ and so by Lemma \ref{Lemma-tan-nor}, one has $T_\mathcal{S}(\bar x)=\text{span}\{e_1, e_2\}$ and 
	$$N_\mathcal{S}(\bar x)=N_\mathcal{S}^C(\bar x)=\text{span}\{e_3\}.$$  
	An easy computation shows that $\nabla f(\bar x)=(0,0,0)^\top$, $\nabla\phi(\bar x)=(1,0,0)^\top$, and so $\bar \lambda=0$ is the unique  Lagrange multiplier of \eqref{Problem} at $\bar x$. Clearly, the RSRCQ holds at $\bar x$ with respect to $\bar \lambda$. Hence
	\begin{align*}
		T_\Omega(\bar x)=T_\Phi(\bar x)\cap \text{span}\{e_1, e_2\}= \{u=(u_1, u_2, u_3)\in\R^3\,:\, u_1=u_3=0, u_2\in\R\}.
	\end{align*}
	Thus
	\begin{align*}
		\partial L^2(\bar x,\bar \lambda)(u)&=\partial f^2(\bar x)(u)=\partial\langle u, \nabla f\rangle(\bar x)
		\\
		&=\{0\}\times\partial(u_2|\cdot-1|)(1)\times\{0\}
		\\
		&=\begin{cases}
			\{0\}\times[-u_2, u_2]\times\{0\}, \ \ &\text{if}\ \ u_2>0,
			\\
			\{0\}\times\{-u_2, u_2\}\times\{0\}, \ \ &\text{if}\ \ u_2<0,
			\\
			\{0_{\R^3}\}, \ \ &\text{otherwise},
		\end{cases}
	\end{align*}
	and so the inequality $\langle z, u\rangle\geq 0$ does not hold for all $z\in \partial L^2(\bar x,\bar \lambda)(u)$ when $u\in T_\Omega(\bar x)\setminus\{0_{\R^3}\}$. We note that when $u=(0, u_2, 0)$ with $u_2<0$, then $\widehat{\partial}^2 L^2(\bar x,\bar \lambda)(u)=\emptyset$. Hence, in this case, the second-order necessary optimality condition (3.7) in \cite[Theorem 3.2]{ChenJW-et-at-2026} has no meaning.
\end{example}
\subsection{Sparse Optimization with Polyhedral Constraints} 

The next result gives second-order necessary optimality conditions for \eqref{Problem} when $\Phi$ is a polyhedral convex set. 

\begin{proposition}\label{Polyhedral-Pro} Assume that $\Phi$ is a polyhedral convex set and $f$ is of class $C^{1,1}$ around $\bar x\in\Omega$ with $\bar x\neq 0$. If $\bar x$ is a local optimal solution of \eqref{Problem} and the RRCQ holds at $\bar x$, then for each $u\in T_\Omega(\bar x)\cap \mathrm{ker}\nabla f(\bar x)$, there exists $z\in \partial^2 f(\bar x)(u)$ such that $\langle z, u\rangle\geq 0$. 
\end{proposition}
\begin{proof} The proof is trivial when $u=0$. Thus, we assume that $0\neq u\in T_\Omega(\bar x)\cap \mathrm{ker}\nabla f(\bar x)$. Then by the RRCQ, one has $u\in T_\Phi(\bar x)\cap T_\mathcal{S}(\bar x)$. Since $\Phi$ is a polyhedral convex set, $T_\Phi(\bar x)=\mathrm{cone}\,(\Phi-\bar x)$. Hence, we can find some $t>0$ and $x\in\Phi$ satisfying $u=t(x-\bar x)$. Let $k_1\in\N$ be such that $\frac{t}{k_1}<1$. Then by the convexity of $\Phi$, one has
	\begin{equation*}
		x_k:=\bar x+\frac{1}{k}u= \bar x+\frac{t}{k}(x-\bar x)=\Big(1-\frac{t}{k}\Big)\bar x+\frac{t}{k} x\in \Phi \ \ \forall k\geq k_1.
	\end{equation*}	  
	
	We claim that $x_k\in  \mathcal{S}$ for all $k\in\N$ and so $x_k\in\Omega$ for all $k\geq k_1$. Indeed, if $\|\bar x\|_0=s$, then by Lemma \ref{Lemma-tan-nor}, one has $T_\mathcal{S}(\bar x)=\R^n_{\bar J}$.  Hence, $\bar x, u\in \R^n_{\bar J}$. Since $\R^n_{\bar J}$ is a linear subspace,   $x_k=\bar x+\frac{1}{k}u\in \R^n_{\bar J}\subset \mathcal{S}$ for all $k\in\N$. We now assume that $\|\bar x\|_0<s$. Then by Lemma \ref{Lemma-tan-nor}, 
	$$T_\mathcal{S}(\bar x)=\bigcup\{\R^n_J\,:\, J\supseteq \bar J, |J|= s\}.$$
	Thus there exist $J_0\supset \bar J$ with $|J_0|=s$ such that $u\in \R^n_{J_0}$. Clearly $\bar x\in \R^n_{J_0}$ and we therefore get $x_k\in \R^n_{J_0}\subset \mathcal{S}$ for all $k\in\N$, as required.
	
	Since $\bar x$ is a local optimal solution to \eqref{Problem} and $x_k\to \bar x$ as $k\to\infty$, we may assume that $f(x_k)-f(\bar x)\geq 0$ for all $k\geq k_1$. Hence
	\begin{equation*}
		0\leq f(x_k)-f(\bar x)=f(\bar x+ \tfrac{1}{k}u)-f(\bar x)- \frac{1}{k}\langle\nabla f(\bar x), u\rangle \ \ \forall k\geq k_1.
	\end{equation*}
	By applying the Taylor formula (Theorem~\ref{Taylor_formula}) for $f$ on $[\bar x, x_k]$ we can find $z_k\in \partial^2f(\xi_k)(u)$ for some $\xi_k\in [\bar x, x_k]$ such that
	\begin{equation*}
		0\leq	f(x_k)-f(\bar x)\leq \frac{1}{2k^2}\langle z_k, u\rangle.
	\end{equation*}
	Hence, $\langle z_k, u\rangle\geq 0$ for all $k\geq k_1$. Since $f$ is of class $C^{1,1}$ around $\bar x$ and $\xi_k\to \bar x$ as $k\to\infty$, we see that the sequence $z_k$ is bounded.   Hence, we may assume that $z_k$ converges to some $z\in\partial^2f(\bar x)(u)$ and so $\langle z, u\rangle\geq 0$. \qed
\end{proof}

The following theorem gives second-order necessary optimality conditions for \eqref{Problem} when $\phi$ is an affine mapping and $D$ is a polyhedral convex set. 

\begin{theorem} Assume that $\phi(x)=Ax+b$ is an affine mapping, where $A\in \R^{m\times n}$ and $b\in\R^m$, $D$ is a polyhedral convex set, $0\neq \bar x\in\Omega$ is a local optimal solution to problem \eqref{Problem}, and $f$ is of class $C^{1,1}$ around $\bar x$.  If the RRCQ holds at $\bar x$, then for each $u\in T_\Omega(\bar x)\cap\mathrm{ker}\nabla f(\bar x)$ and $\bar\lambda\in \Lambda^C(\bar x)$, there exists $z\in \partial^2 L(\bar x, \bar \lambda)(u)$ such that $\langle z, u\rangle\geq 0$.
\end{theorem}
\begin{proof} By \cite[Theorem 19.3]{Rockafellar1970}, we see that $\Phi=A^{-1}(D-b)$ is a polyhedral convex set. Now let $u\in T_\Omega(\bar x)\cap\mathrm{ker}\nabla f(\bar x)$ and $\bar\lambda\in \Lambda^C(\bar x)$. Then, we have
	\begin{equation*}
		\nabla_x L(x, \bar\lambda)=\nabla f(x)+ A^\top \bar\lambda \ \ \forall x\in\R^n.
	\end{equation*}
	Hence
	\begin{align*}
		\partial^2L(\bar x, \bar \lambda)(u)&=\partial\langle u, \nabla_x L(\cdot, \bar\lambda)\rangle(\bar x)= \partial\langle u, \nabla f(\cdot)+A^\top \bar\lambda\rangle(\bar x)
		\\
		&= \partial\langle u, \nabla f(\cdot)\rangle(\bar x)=\partial^2 f(\bar x)(u).
	\end{align*}
	and the conclusion follows directly from Proposition \ref{Polyhedral-Pro}. \qed
\end{proof}

By using the radial cone of $\Phi$, we can derive second-order necessary optimality conditions for \eqref{Problem} when $\phi$ is an affine mapping and $D$ is a closed and convex cone, not necessarily polyhedral.

\begin{theorem} Assume that $\phi(x)=Ax+b$ is an affine mapping, where $A\in \R^{m\times n}$ and $b\in\R^m$, $D$ is a closed and convex cone, $0\neq \bar x\in\Omega$ is a local optimal solution to problem \eqref{Problem}, and $f$ is of class $C^{1,1}$ around $\bar x$.  Then for each $u\in \mathcal{R}_\Phi(\bar x)\cap T_\mathcal{S}(\bar x) \cap \mathrm{ker}\nabla f(\bar x)$, there exists $z\in \partial^2 f(\bar x)(u)$ such that $\langle z, u\rangle\geq 0$.
\end{theorem}
\begin{proof} Since $u\in \mathcal{R}_\Phi(\bar x)$, there exists $\tau>0$ such that $\bar x+ tu\in \Phi$ for all $t\in [0, \tau]$. This implies that $x_k:=\bar x+\frac{1}{k}u\in \Phi$ for all $k$ large enough.  An argument similar to that used in the proof Proposition \ref{Polyhedral-Pro} shows that $x_k\in \mathcal{S}$ and so does $x_k\in \Omega$ for all $k$ large enough. Hence, we can repeat the final part of the proof of Proposition \ref{Polyhedral-Pro} to get the conclusion. \qed
\end{proof}
\subsection{Sparse Optimization with Mixed Constraints}
Consider problem \eqref{Problem} with $\phi=(g, h)$ and $D=\R^p_-\times\{0_{\R^q}\}$, where $g\colon\R^n\to\R^p$ and $h\colon\R^n\to \R^q$ are vector-valued functions. Then
\begin{equation}\label{Mixed-constraint} 
	\Phi=\{x\in\R^n\,:\, g_i(x)\leq 0, i=1, \ldots, p,\, h_j(x)=0, j=1, \ldots, q\},
\end{equation}
$\Omega=\Phi\cap\mathcal{S}$, and  the Lagrangian function associated with problem \eqref{Problem} is defined as
\begin{equation*}
	L(x, \lambda, \mu):=f(x)+\langle \lambda, g(x)\rangle+\langle \mu, h(x)\rangle, \ \ \forall x\in\R^n, \lambda\in\R^p_+, \mu\in\R^q.  
\end{equation*}

Let $\bar x\in\Omega$. The active index set to $\bar x$ is defined by
\begin{equation*}
	I(\bar x):=\{i\in \{1, \ldots, p\}\,:\, g_i(\bar x)=0\}.
\end{equation*}
Then the sets of Lagrange multipliers at $\bar x$ in the sense of Fr\'echet and Clarke  are given, respectively, by
\begin{align*}
	\Lambda^S(\bar x)&:=\{(\lambda, \mu)\in \R^p_+\times\R^q :  -\nabla_x L(\bar x, \lambda, \mu)\in \widehat{N}_\mathcal{S}(\bar x), \lambda_i g_i(\bar x)=0, i=1, \ldots, p\},
	\\
	\Lambda^C(\bar x)&:=\{(\lambda, \mu)\in \R^p_+\times\R^q :  -\nabla_x L(\bar x, \lambda, \mu)\in {N}^C_\mathcal{S}(\bar x), \lambda_i g_i(\bar x)=0, i=1, \ldots, p\}.
\end{align*} 

If $\Lambda^S(\bar x)$ (resp., $\Lambda^C(\bar x)$) is nonempty, then $\bar x$ is called an S-stationary point (resp., C-stationary point) of problem \eqref{Problem}.
\begin{definition}[{\normalfont  see \cite[Definition 2.4]{Pan-Xiu-Fan}}]\rm  We say that:
	\begin{enumerate}[\rm(i)]
		\item The {\em  restricted linear independence constraint qualification} (R-LICQ) holds at $\bar x$ if
		
		$\bullet$ when $\|\bar x\|_0=s$, $\nabla g_i(\bar x)$, $i\in I(\bar x)$, $\nabla h_j(\bar x)$, $j=1, \ldots, q$, are linearly independent;
		
		$\bullet$ when  $0<\|\bar x\|_0<s$, $\nabla_{\bar J}\, g_i(\bar x)$, $i\in I(\bar x)$, $\nabla_{\bar J}\, h_j(\bar x)$, $j=1, \ldots, q$, are linearly independent, where $\nabla_{\bar J}\, g_i(\bar x):=(\nabla g_i(\bar x))_{\bar J}$ and $\nabla_{\bar J}\, h_j(\bar x):=(\nabla h_j(\bar x))_{\bar J}$. 
		
		\item  The {\em  restricted Mangasarian--Fromovitz constraint qualification} (R-MFCQ) holds at $\bar x$ if
		
		$\bullet$ when $\|\bar x\|_0=s$, $\nabla_{\bar J}\, h_j(\bar x)$, $j=1, \ldots, q$, are linearly independent, and there exist a vector $y\in\R^n$ such that
		\begin{equation*}
			\langle\nabla g_i(\bar x), y\rangle <0, i\in I(\bar x) \ \ \text{and}\ \ \langle\nabla h_j(\bar x), y\rangle =0, j=1, \ldots, q;
		\end{equation*} 
		
		$\bullet$ when  $0<\|\bar x\|_0<s$, $\nabla_{\bar J}\, h_j(\bar x)$, $j=1, \ldots, q$, are linearly independent, and for any $J\in  \mathcal{J}(\bar x):=\{J\in\mathcal{J}\,:\, J\supset \bar J\}$, there exists a vector $y\in \R^n_J$ such that
		\begin{equation*}
			\langle\nabla g_i(\bar x), y\rangle <0, i\in I(\bar x) \ \ \text{and}\ \ \langle\nabla h_j(\bar x), y\rangle =0, j=1, \ldots, q.
		\end{equation*}
	\end{enumerate}
\end{definition}

\begin{remark} \rm 
	\begin{enumerate}[\rm(i)]
		\item Clearly, if the R-LICQ holds at $\bar x$ then  so does R-MFCQ. Furthermore,  when $\|\bar x\|_0=s$, these conditions reduce to  the classical LICQ and the classical MFCQ, respectively. 
		
		\item It is clear that if the R-LICQ (resp., R-MFCQ) holds at $\bar x$ then  so does the LICQ (resp.,  MFCQ). 
		
		\item (see \cite[Remark 1]{Zhao-et-al}) If $\|\bar x\|_0=s$ and the RRCQ holds at $\bar x$, then so does the R-LICQ.
	\end{enumerate}		
\end{remark}

\begin{lemma}[{\normalfont see \cite[Proposition 2.5, Proposition 2.11 and Remark 2.6]{Pan-Xiu-Fan} and \cite[Exercise 2.53]{Mordukhovich_2018}}]\label{Tang-Inter} Assume that $0\neq \bar x\in\Omega$. 
	\begin{enumerate}[\rm(i)]
		\item If the R-MFCQ holds at $\bar x$, then $T_\Omega(\bar x)=T_\Phi(\bar x)\cap T_\mathcal{S}(\bar x)$ and $${N}_\Omega(\bar x)\subset {N}_\Phi(\bar x)+ {N}_\mathcal{S}(\bar x).$$ 
		Furthermore, we have
		\begin{equation*}
			\widehat{N}_\Phi(\bar x)= {N}_\Phi(\bar x)= \bigg\{\sum_{i\in I(\bar x)}\lambda_i\nabla g_i(\bar x)+\sum_{j=1}^q\mu_j\nabla h_j(\bar x)\,:\, \lambda\in\R^{|I(\bar x)|}_+, \mu\in\R^q\bigg\}.
		\end{equation*} 
		\item If the R-LICQ  holds at $\bar x$, then $\widehat{N}_\Omega(\bar x)=\widehat{N}_\Phi(\bar x)+ \widehat{N}_\mathcal{S}(\bar x)$ and 
		$${N}_\Omega(\bar x)={N}_\Phi(\bar x)+ {N}_\mathcal{S}(\bar x).$$  
	\end{enumerate}	
\end{lemma}  

\begin{lemma}[{\normalfont see \cite[Theorems 3.2 and 3.5]{Pan-Xiu-Fan}}]\label{Nonemty-Lagran} Assume that $0\neq \bar x\in\Omega$ is a local optimal solution to \eqref{Problem}.
	\begin{enumerate}[\rm(i)]
		\item If the R-MFCQ holds at $\bar x$, then $\Lambda^C(\bar x)$ is nonempty.
		\item  If the R-LICQ  holds at $\bar x$, then $\Lambda^S(\bar x)$ is nonempty.
	\end{enumerate}
\end{lemma}

The following example shows that the R-LICQ at $\bar x$  is not sufficient to guarantee that $\Lambda^S(\bar x)$ is a singleton.   
\begin{example}\rm Consider problem
	\begin{equation*}
		\min_{\R^3} f(x):=(x_1-1)^2+(x_2-1)^2 \ \ \text{s.t.}\ \ x\in \Omega:=\{x\in\Phi\,:\, \|x\|_0\leq 2\},
	\end{equation*}	 
	where 
\begin{align*}
\Phi:=\big\{x\in \R^3\,:\, g_1(x):=x_1-1\leq 0, \ \ &g_2(x):=x_2-1\leq 0,
\\
 &h(x):=x_1+x_2+x_3-2=0\big\}.
\end{align*}	
Then $\bar x=(1,1,0)\in\Omega$ is a global optimal solution of   \eqref{Problem}, $\nabla f(\bar x)=(0, 0, 0)^\top$, $\nabla g_1(\bar x)=(1, 0, 0)^\top$, $\nabla g_2(\bar x)=(0, 1, 0)^\top$, and $\nabla h(\bar x)=(1, 1, 1)^\top$. It is easy to check that $\|\bar x\|_0=2$, $I(\bar x)=\{1, 2\}$, $\bar J=\{1, 2\}$, and the R-LICQ holds at $\bar x$.  Hence, $\widehat{N}_\mathcal{S}(\bar x)=\R^3_{{\bar J}^c}=\mathrm{span}\{e_3\}$. Thus
	\begin{align*}
		\Lambda^S(\bar x)=\big\{((\lambda_1, \lambda_2), \mu)\in\R^2_+\times\R\,:\, \nabla f(\bar x)&+\lambda_1\nabla g_1(\bar x)+\lambda_2\nabla g_2(\bar x)
		\\
		&\ \ \ \ \ \ \ \ \ \ \ \ \  +\mu\nabla h(\bar x)\in \mathrm{span}\{e_3\}\big\}
		\\
		=\{((\lambda_1, \lambda_2), \mu)\in\R^2_+\times\R\,:\,\lambda_1=\lambda_2&=-\mu, \mu\in\R\}=\mathrm{span}\{(1,1,-1)\}. 
	\end{align*}
\end{example}

\medskip 

For $\bar x\in \Omega$ and $(\bar\lambda, \bar\mu)\in \Lambda^S(\bar x)$, we put $\bar I(\bar x):=\{i\in \{1, \ldots, p\}\,:\, \lambda_i>0\}$. Clearly, $\bar I(\bar x)\subset I(\bar x)$. Let 
\begin{equation*}
	\bar\Phi:=\Phi\cap \{x\in\R^n\,:\, g_i(x)=0, i\in \bar I(\bar x)\}  
\end{equation*}
and 
\begin{align*}
	L_{\bar\Phi}(\bar x):=\big\{u\in\R^n\,:\, \langle\nabla g_i(\bar x), u\rangle=0, i\in \bar I(\bar x),  &\langle\nabla g_i(\bar x), u\rangle\leq 0, i \in I(\bar x)\setminus \bar I(\bar x),  
	\\
	&\langle\nabla h_j(\bar x), u\rangle=0,\ \ j=1, \ldots, q\big\}.
\end{align*}

We now present second-order necessary optimality conditions for \eqref{Problem} with mixed constraints.

\begin{theorem}\label{Necessary-Mixed} Let $0\neq \bar x\in \Omega$. Suppose that $f, g,$ and  $h$ are of class $C^{1,1}$ around $\bar x$, and the R-LICQ holds at $\bar x$.  
	\begin{enumerate}[\rm(i)]
		\item If $\bar x$ is a local optimal solution to problem \eqref{Problem}, then for each $(\bar\lambda, \bar \mu)\in\Lambda^S(\bar x)$ and $u\in  T_{\bar \Phi}(\bar x)\cap T_\mathcal{S}(\bar x)$, there exists $z\in \partial^2 L(\bar x, \bar\lambda, \bar\mu)(u)$ such that $\langle z, u\rangle\geq 0$. 
		
		\item If $\bar x$ is an isolated local  solution of order $2$ to problem \eqref{Problem}, then for each $(\bar\lambda, \bar \mu)\in\Lambda^S(\bar x)$ and $u\in[T_{\bar \Phi}(\bar x)\cap T_\mathcal{S}(\bar x)]\setminus\{0\}$, there exists $z\in \partial^2 L(\bar x, \bar\lambda, \bar\mu)(u)$ such that $\langle z, u\rangle > 0$. 
		
	\end{enumerate}	
\end{theorem}
\begin{proof} If the R-LICQ holds for the constraint system $\Phi\cap \mathcal{S}$ at $\bar x$, then $\Lambda^S(\bar x)$ is nonempty. Let $(\bar\lambda, \bar \mu)\in\Lambda^S(\bar x)$ and $u\in  T_{\bar \Phi}(\bar x)\cap T_\mathcal{S}(\bar x)$. Since $\bar\Phi\subset \Phi$, the R-LICQ holds  for the constraint system $\bar\Phi\cap \mathcal{S}$ at $\bar x$. By Lemma \ref{Tang-Inter}, one has $T_{\bar \Phi\cap \mathcal{S}}(\bar x)= T_{\bar \Phi}(\bar x)\cap T_\mathcal{S}(\bar x)$ and so $u\in T_{\bar \Phi\cap \mathcal{S}}(\bar x)$. Hence there exist sequences $t_k\to 0^+$ and $u_k\to u$ as $k\to \infty$ such that $x_k:=\bar x+t_ku_k\in \bar \Phi\cap \mathcal{S}$   for all $k\in\N$.    If $\bar x$ is a local optimal solution to problem \eqref{Problem}, then
	\begin{align*}
		L(x_k, \bar\lambda, \bar\mu)-L(\bar x, \bar\lambda, \bar\mu)&=f(x_k)-f(\bar x)+ \sum_{i=1}^p\bar\lambda(g_i(x_k)-g_i(\bar x))
		\\
		& \ \ \ \ \ \ \ \ \ \ \ \ \ \ \ \ \ \ \ \ \ \ +\sum_{j=1}^q\mu_j(h_j(x_k)-h_j(\bar x))
		\\
		&=f(x_k)-f(\bar x)\geq 0
	\end{align*} 
	for all $k$ large enough. An argument similar to that used in the proof Theorem \ref{Necessary-equality} shows that 
	$$\langle \nabla_x L(\bar x, \bar\lambda, \bar\mu), u_k\rangle= \langle \nabla_x L(\bar x, \bar\lambda, \bar\mu), u\rangle =0$$
	for all $k$ large enough. Hence, by repeating the final part of the proof of Theorem   \ref{Necessary-equality},  we can find some $z\in \partial^2 L(\bar x, \bar\lambda, \bar\mu)(u)$ such that $\langle z, u\rangle\geq 0$. 
	
	When $\bar x$ is an isolated local  solution of order $2$ to problem \eqref{Problem}, the proof follows along the same lines as that of Theorem \ref{Necessary-equality}(iii) and is therefore omitted. \qed
\end{proof}

Since the RRCQ implies the R-LICQ, and the uniqueness of the Lagrange multiplier $\bar{\lambda}$ in $\Lambda(\bar{x})$ does not play any role in the proof of Theorem \ref{Necessary-equality}, we obtain the following consequence  of Theorem \ref{Necessary-Mixed}. It notes that $\bar \Phi=\Phi$ when $D=\{0_{\R^m}\}$.    
\begin{corollary}  Suppose that $\bar x\neq 0$ is a local optimal solution to \eqref{Problem} with $D=\{0_{\R^m}\}$,  and $f, \phi$ are of class $C^{1,1}$ at $\bar x$. If the R-LICQ holds at $\bar x$, i.e.,
	
	$\bullet$ the system $\{\nabla \phi_k(\bar x)\,:\, k=1, \ldots, m\}$, are linearly independent when $\|\bar x\|_0=s$;
	
	$\bullet$ the system $\{\nabla_{\bar J}\, \phi_k(\bar x)\,:\, k=1, \ldots, m\}$, are linearly independent when  $0<\|\bar x\|_0<s$,	
	\\
	then for each $\bar\lambda\in \Lambda^S(\bar x)$ and $u\in T_\Omega(\bar x)$, there exists $z\in \partial^2 L(\bar x, \bar\lambda, \bar\mu)(u)$ such that $\langle z, u\rangle\geq 0$.
\end{corollary}
\section{Second-Order Sufficient Optimality Conditions}\label{Section 4}
In this section, we present some second-order sufficient optimality conditions for problem \eqref{Problem} with mixed constraints, i.e., 
\begin{equation*}
	\Phi=\{x\in\R^n\,:\, g_i(x)\leq 0, i=1, \ldots, p,\, h_j(x)=0, j=1, \ldots, q\}
\end{equation*}
and $\Omega=\Phi\cap\mathcal{S}$. We assume that $f, g,$ and  $h$ are of class $C^{1,1}$ around the considered feasible solution $\bar x\in\Omega$. The {\em  cone of critical directions}  at $\bar x$ is defined by
\begin{align*}
	\mathcal{C}(\bar x):=\big\{u\in\R^n\,:\, \langle\nabla f(\bar x), u\rangle=0,\ \  &\langle\nabla g_i(\bar x), u\rangle\leq0, i\in I(\bar x), 
	\\
	&\langle\nabla h_j(\bar x), u\rangle=0, j=1, \ldots, q\big\}.
\end{align*}

\begin{theorem}\label{Sufficient-Theorem-1} Let $\bar x\in \Omega$ be an S-stationary point with $(\bar \lambda, \bar\mu)\in \Lambda^S(\bar x)$. If for every $u\in [T_\Omega(\bar x)\cap \mathcal{C}(\bar x)]\setminus\{0_{\R^n}\}$ and for any $z\in \partial^2 L(\bar x, \bar\lambda, \bar\mu)$ one has
	\begin{equation*}
		\langle z, u\rangle>0,
	\end{equation*}
	then $\bar x$ is an isolated local solution of order $2$ of problem \eqref{Problem}.	
\end{theorem}
\begin{proof} Suppose to the contrary that $\bar x$ is not an isolated local solution of order $2$ of problem \eqref{Problem}. Then, there exists a sequence $x_k\in \Omega\setminus\{\bar x\}$ such that $x_k\to \bar x$ as $k\to\infty$ and
	\begin{equation}\label{New-equa-3}
		f(x_k)-f(\bar x)< \frac{1}{k}\|x_k-\bar x\|^2 \ \ \forall k\in\N.
	\end{equation}	 
	Put $t_k:=\|x_k-\bar x\|$ and $u_k:=\frac{x_k-\bar x}{\|x_k-\bar x\|}$, then $\|u_k\|=1$ for all $k\in\N$ and $t_k\to 0^+$ as $k\to\infty$.  By the boundedness of $u_k$, without loss of generality we may assume that $u_k$ converges to some $u\in\R^n$ with $\|u\|=1$. Clearly, $u\in T_\Omega(\bar x)$ and so $u\in T_\Phi(\bar x)\cap T_\mathcal{S}(\bar x)$. We claim that $u\in\mathcal{C}(\bar x)$. Indeed, since $x_k\in\Omega$, for each $i\in I(\bar x)$ one has
	\begin{equation*}
		0\geq g_i(x_k)=g_i(\bar x+t_ku_k)-g_i(\bar x)=t_k\langle \nabla g_i(\theta_k), u_k\rangle 
	\end{equation*}
	for some $\theta_k\in (\bar x, x_k)$. Dividing both sides by $t_k$ and then letting $k\to\infty$, we obtain $\langle \nabla g_i(\bar x), u\rangle\leq 0$ for $i\in I(\bar x)$. A similar argument applied to the equality constraints yields $\langle \nabla h_j(\bar x), u\rangle= 0$ for all  $j=1, \ldots, q$. On the other hand, it follows from \eqref{New-equa-3} that for each $k\in\N$, one has
	\begin{equation*}
		t_k\langle \nabla f(\xi_k), u_k\rangle< \frac{1}{k}t_k^2
	\end{equation*}
	for some $\xi_k\in (\bar x, x_k)$. Dividing both sides by $t_k$ and  letting $k\to\infty$, we obtain $\langle \nabla f(\bar x), u\rangle\leq 0$. We now show  that $\langle \nabla f(\bar x), u\rangle\geq 0$ and so $\langle \nabla f(\bar x), u\rangle= 0$. To this end, it suffices to show that $\langle\nabla_x L(\bar x, \bar \lambda, \bar\mu), u\rangle=0$. Indeed, if it is this case, then we have
	\begin{equation*}
		\langle \nabla f(\bar x), u\rangle= - \sum_{i\in I(\bar x)}\bar\lambda_i\langle \nabla g_i(\bar x), u\rangle-\sum_{j=1}^q\bar\mu_j\langle \nabla h_j(\bar x), u\rangle\geq 0, 
	\end{equation*}
	as required. Since $\bar x$ is an S-stationary, one has $-\nabla_x L(\bar x, \bar \lambda, \bar\mu)\in \widehat{N}_\mathcal{S}(\bar x)$. If $\|\bar x\|_0=s$, then  $\widehat{N}_\mathcal{S}(\bar x)=\R^n_{{\bar J}^c}$ and $T_\mathcal{S}(\bar x)=\R^n_{\bar J}$. This, together with the fact that $u\in T_\mathcal{S}(\bar x)$, implies that $\langle\nabla_x L(\bar x, \bar \lambda, \bar\mu), u\rangle=0$. If $\|\bar x\|_0<s$, then $\widehat{N}_\mathcal{S}(\bar x)=\{0_{\R^n}\}$ and so   $\langle\nabla_x L(\bar x, \bar \lambda, \bar\mu), u\rangle=0$. Hence  $\langle \nabla f(\bar x), u\rangle= 0$ and so $u\in\mathcal{C}(\bar x)$.
	
	Now, by the definition of the Lagrangian function and \eqref{New-equa-3}, we have
	\begin{align}
		L(x_k, \bar \lambda, \bar\mu)-L(\bar x, \bar \lambda, \bar\mu)=f(x_k)-f(\bar x)&+\sum_{i\in I(\bar x)}\bar\lambda_i[g_i(x_k)-g_i(\bar x)]\notag
		\\
		&+ \sum_{j=1}^q\bar\mu_j[h_j(x_k)-h_j(\bar x)]<\frac{1}{k}t_k^2 \label{New-equa-4}
	\end{align}
	for all $k\in\N$. Since $\langle\nabla_x L(\bar x, \bar \lambda, \bar\mu), u\rangle=0$, one has
	\begin{align*}
		L(x_k, \bar \lambda, \bar\mu)-L(\bar x, \bar \lambda, \bar\mu)&=[L(x_k, \bar \lambda, \bar\mu)-L(\bar x+t_ku, \bar \lambda, \bar\mu)]
		\\
		&+[L(\bar x+t_ku, \bar \lambda, \bar\mu)- L(\bar x, \bar \lambda, \bar\mu)- t_k\langle\nabla_x L(\bar x, \bar \lambda, \bar\mu), u\rangle].
	\end{align*} 
	Put $L_{1k}:=L(x_k, \bar \lambda, \bar\mu)-L(\bar x+t_ku, \bar \lambda, \bar\mu)$ and  
	\begin{equation*}
		L_{2k}:=L(\bar x+t_ku, \bar \lambda, \bar\mu)- L(\bar x, \bar \lambda, \bar\mu)- t_k\langle\nabla_x L(\bar x, \bar \lambda, \bar\mu), u\rangle.
	\end{equation*}
	An argument similar to that used in the proof Theorem \ref{Necessary-equality} shows that $\frac{L_{1k}}{t_k^2}\to 0$ when $k\to\infty$.  By applying the Taylor formula (Theorem \ref{Taylor_formula}) for the function $L(\cdot, \bar\lambda, \bar \mu)$ on $[\bar x, \bar x+t_ku]$, we can find $z_k\in \partial^2 L(\eta_k, \bar\lambda, \bar\mu)(u)$ for some $\eta_k\in [\bar x, \bar x+t_ku]$ such that
	\begin{equation*}
		L_{2k}\geq \frac{1}{2}t_k^2 \langle z_k, u\rangle
	\end{equation*} 
	and $z_k$ converges to some $z\in \partial^2 L(\bar x, \bar\lambda, \bar\mu)(u)$. This, together with \eqref{New-equa-4}, implies that
	\begin{equation*}
		\frac{1}{2}t_k^2 \langle z_k, u\rangle +L_{1k}\leq L_{1k}+L_{2k}\leq \frac{1}{k}t_k^2 \ \ \forall k\in\N.	
	\end{equation*}
	Dividing both sides by $\frac{1}{2}t^2_k$ and then letting $k\to\infty$, we obtain $\langle z, u\rangle\leq 0$, a contradiction. Hence $\bar x$ is an isolated local solution of order $2$ of problem \eqref{Problem}.  \qed
\end{proof}

Let us consider an example to illustrate  Theorem \ref{Sufficient-Theorem-1}.

\begin{example}\rm Consider problem \eqref{Problem} with $n=3$, $f(x)=x_1+(x_2-1)|x_2-1|$, $s=2$, 
\begin{align*}
\Phi:=\{x\in \R^3\,:\, g_1(x):=-x_1+1\leq 0,\ \ &g_2(x):=-x_2+1\leq 0,
\\
& h(x):= x_1+x_2+x_3-2=0\},
\end{align*}	
and $\Omega=\Phi\cap\mathcal{S}$. Let $\bar x=(1,1,0)\in\Omega$. Then, $\|\bar x\|_0=2$, $I(\bar x)=\{1,2\}$, $\bar J=\Gamma(\bar x)=\{1, 2\}$,   $\nabla f(\bar x)=(1, 0, 0)^\top$, $\nabla g_1(\bar x)= (-1, 0, 0)^\top$, $\nabla g_2(\bar x)= (0, -1, 0)$, and $\nabla h(\bar x)=(1,1,1)^\top$. Hence, $\widehat{N}(\bar x)=\R^3_{{\bar J}^c}=\text{span}\{e_3\}$ and $\{\nabla g_1(\bar x), \nabla g_2(\bar x), \nabla h(\bar x)\}$ are linear independent. This means that the R-LICQ holds at $\bar x$ and so $\Lambda^S(\bar x)$  is nonempty. By the R-LICQ, one has 
	\begin{align*}
		&T_\Omega(\bar x)=T_\Phi(\bar x)\cap T_\mathcal{S}(\bar x)=T_\Phi(\bar x)\cap \R^3_{\bar J}=T_\Phi(\bar x)\cap \text{span}\{e_1, e_2\}
		\\
		&=\{u\in\R^3\,:\, \langle\nabla g_1(\bar x), u\rangle\leq 0, \langle\nabla g_2(\bar x), u\rangle\leq 0,  \langle\nabla h(\bar x), u\rangle= 0\}\cap \text{span}\{e_1, e_2\}
		\\
		&=\{u\in\R^3\,:\, -u_1\leq 0, -u_2\leq 0, u_1+u_2+u_3=0\}\cap \text{span}\{e_1, e_2\}=\{0_{\R^3}\}.
	\end{align*}	
	Hence $[T_\Omega(\bar x)\cap \mathcal{C}(\bar x)]\setminus\{0_{\R^3}\}=\emptyset$. This, together with Theorem \ref{Sufficient-Theorem-1}, implies that $\bar x$ is an isolated local solution of order $2$ of \eqref{Problem}. We note here that since $f$ is not of class $C^2$ at $\bar x$, Theorem 4.2 in \cite{Pan-Xiu-Fan} cannot be applied for this example. 	
\end{example}
\section{Application to Sparse Multiobjective Optimization}\label{Section 5}
In this section, we consider the following sparse multiobjective optimization problem:
\begin{equation}\label{SMP}
	\mathrm{Min}_{\R^d_+} F(x):=(f_1(x), \ldots, f_d(x)) \ \ \text{s.t.}\ \ x\in\Omega:=\Phi\cap\mathcal{S},\tag{SMP} 
\end{equation} 
where $\Phi$ is defined as in \eqref{Mixed-constraint}. The Lagrangian function associated with problem \eqref{SMP} is defined as
\begin{equation*}
	L^M(x, \alpha, \lambda, \mu):=\langle \alpha, F(x)\rangle+\langle \lambda, g(x)\rangle+\langle \mu, h(x)\rangle, 
\end{equation*} 
for all $x\in\R^n,$ $\alpha\in\R^d_+\setminus\{0_{\R^d}\},$ $\lambda\in\R^p_+,$ and $\mu\in\R^q.$

\medskip
  
We assume that $F, g,$ and $h$ are of class $C^{1,1}$ around $\bar x$.
\begin{definition}\rm For $\bar x\in\Omega$, put 
	\begin{equation*} 
		\varphi(x):=\max\{\varphi_l(x)\,:\, l\in\{1, \ldots, d\}\}\ \ \forall x\in\R^n,	
	\end{equation*}
	where $\varphi_l(x):=f_l(x)-f_l(\bar x)$. We say that  $\bar x$ is an {\em isolated local  efficient solution of order $2$}  of \eqref{SMP} if $\bar x$ is an isolated local   solution of order $2$ of $\varphi$ on $\Omega$. 	
\end{definition}

Then the sets of {\em S-Lagrange multipliers} at $\bar x$ is defined by
\begin{align*}
	\mathrm{A}^S(\bar x):=\Big\{(\alpha,\lambda, \mu)\in (\R^d_+\setminus\{0_{\R^d}\})\times \R^p_+\times\R^q\,:\, &-\nabla_x L(\bar x, \alpha,  \lambda, \mu)\in \widehat{N}_\mathcal{S}(\bar x), 
	\\
	& \ \ \ \ \ \lambda_i g_i(\bar x)=0, i=1, \ldots, p\Big\}.	
\end{align*}
If $\mathrm{A}^S(\bar x)$ is nonempty, then $\bar x$ is called an {\em S-stationary point} of \eqref{SMP} at $\bar x$.

We say that $u\in \R^n$ is a {\em   critical direction} of \eqref{SMP}  at $\bar x$  if
\begin{equation*}
	\begin{cases}
		&\langle\nabla f_l(\bar x), u\rangle\leq 0, l=1, \ldots, d,
		\\
		&\langle\nabla f_l(\bar x), u\rangle=0 \ \ \text{for at least one}\ \ l\in \{l=1, \ldots, d\},
		\\
		&\langle\nabla g_i(\bar x), u\rangle\leq 0, i\in I(\bar x),
		\\
		&\langle\nabla h_j(\bar x), u\rangle=0, j=1, \ldots, q.
	\end{cases}
\end{equation*}

The set of critical directions of \eqref{SMP} at $\bar x$ is denoted by $\mathcal{K}(\bar x)$. 

The following theorem gives a second-order sufficient optimality condition for an isolated local  efficient solution of order $2$ of problem \eqref{SMP}.

\begin{theorem}\label{Sufficient-Theorem-2} Let $\bar x\in \Omega$ be an S-stationary point with $(\bar \alpha, \bar \lambda, \bar\mu)\in \mathrm{A}^S(\bar x)$. If for every $u\in [T_\Omega(\bar x)\cap \mathcal{K}(\bar x)]\setminus\{0_{\R^n}\}$ and for any $z\in \partial^2 L^M(\bar x, \bar\alpha, \bar\lambda, \bar\mu)$ one has
	\begin{equation*}
		\langle z, u\rangle>0,
	\end{equation*}
	then $\bar x$ is an isolated local solution of order $2$ of problem \eqref{SMP}.	
\end{theorem}
\begin{proof} Suppose to the contrary that $\bar x$ is not an isolated local solution of order $2$ of problem \eqref{SMP}. Then, there exists a sequence $x_k\in \Omega\setminus\{\bar x\}$ such that $x_k\to \bar x$ as $k\to\infty$ and
	\begin{equation*}
		\varphi(x_k)=\max\{f_l(x_k)-f_l(\bar x)\,:\, l\in\{1, \ldots, d\}\}< \frac{1}{k}\|x_k-\bar x\|^2 \ \ \forall k\in\N.
	\end{equation*}	
	This implies that
	\begin{equation}\label{New-equa-5}
		f_l(x_k)-f_l(\bar x)< \frac{1}{k}\|x_k-\bar x\|^2 \ \ \forall l=1, \ldots, d, \forall k\in\N.	
	\end{equation}
	
	Put $t_k:=\|x_k-\bar x\|$ and $u_k:=\frac{x_k-\bar x}{\|x_k-\bar x\|}$, then $\|u_k\|=1$ for all $k\in\N$ and $t_k\to 0^+$ as $k\to\infty$.  Since  $\|u_k\|=1$ for all $k\in\N$, without loss of generality we may assume that $u_k$ converges to some $u\in\R^n$ with $\|u\|=1$. Then, $u\in T_\Omega(\bar x)$ and so $u\in T_\Phi(\bar x)\cap T_\mathcal{S}(\bar x)$. We claim that $u\in\mathcal{K}(\bar x)$. Indeed, analysis similar to that in the proof of Theorem \ref{Sufficient-Theorem-1} shows that   $\langle \nabla g_i(\bar x), u\rangle\leq 0$ for $i\in I(\bar x)$, and $\langle \nabla h_j(\bar x), u\rangle= 0$ for all  $j=1, \ldots, q$. Furthermore, by \eqref{New-equa-5}, it is easy to see that  $\langle \nabla f_l(\bar x), u\rangle\leq 0$ for all $l=1, \ldots, d$. We now show  that $\langle \nabla f_l(\bar x), u\rangle = 0$ for at least one $l\in\{1, \ldots, d\}$ and so $u\in\mathcal{K}(\bar x)$. If otherwise, then $\langle \nabla f_l(\bar x), u\rangle< 0$ for all $l=1, \ldots, d$. This, together with the fact that $\bar \alpha\in \R^d_+\setminus\{0_{\R^d}\}$, implies that $\sum_{l=1}^d\alpha_l\langle \nabla f_l(\bar x), u\rangle<0$.  Since $\bar x$ is an S-stationary point, analysis similar to that in the proof of Theorem \ref{Sufficient-Theorem-1} shows that  $\langle\nabla_x L^M(\bar x, \bar\alpha, \bar \lambda, \bar\mu), u\rangle=0$. Hence 
	
	\begin{equation*}
		0>\sum_{l=1}^d\alpha_l\langle \nabla f_l(\bar x), u\rangle= - \sum_{i\in I(\bar x)}\bar\lambda_i\langle \nabla g_i(\bar x), u\rangle-\sum_{j=1}^q\bar\mu_j\langle \nabla h_j(\bar x), u\rangle\geq 0, 
	\end{equation*}
	a contradiction.
	
	Now, by the definition of the Lagrangian function and \eqref{New-equa-5}, we have
	\begin{align*}
		L(x_k, \bar\alpha, \bar \lambda, \bar\mu)-L(\bar x,\bar\alpha, \bar \lambda, \bar\mu)=\sum_{l=1}^d&\bar\alpha_l[f_l(x_k)-f_l(\bar x)]+\sum_{i\in I(\bar x)}\bar\lambda_i[g_i(x_k)-g_i(\bar x)]\notag
		\\
		&+ \sum_{j=1}^q\bar\mu_j[h_j(x_k)-h_j(\bar x)]<\frac{1}{k}\Big(\sum_{l=1}^d\bar\alpha_l\Big)t_k^2  
	\end{align*}
	for all $k\in\N$. An argument similar to that used in the proof Theorem \ref{Sufficient-Theorem-1} shows that there exists $z\in \partial^2 L^M(\bar x, \bar\alpha, \bar\lambda, \bar\mu)$  such that   $\langle z, u\rangle\leq 0$, contradicting the assumption. Hence $\bar x$ is an isolated local solution of order $2$ of problem \eqref{SMP}. The proof is complete.  \qed
\end{proof}

\section{Conclusions}\label{Section 6}

In this paper, we investigated second-order optimality conditions for sparse optimization problems with $C^{1,1}$ data. By employing the limiting (Mordukhovich) second-order subdifferential of the associated Lagrangian function, we established new second-order necessary and sufficient optimality conditions for the considered problems. In contrast to approaches based on Fr\'echet second-order subdifferentials, the present framework is applicable to a broader class of nonsmooth problems since the limiting second-order subdifferential of $C^{1,1}$ functions is always nonempty and compact. As an application, we also derived second-order sufficient optimality conditions for sparse multiobjective optimization problems. Several examples were presented to illustrate the applicability and effectiveness of the obtained results.

Second-order optimality conditions for sparse optimization problems, and even for optimization problems without sparse constraints, with merely $C^1$ data remain open. Dealing with this challenging problem requires new ideas and techniques. We plan to investigate this issue in future work.

\section*{Acknowledgments}  A part of this work was done while third author was visiting Academy for Advanced Interdisciplinary Studies, Northeast Normal University, Changchun, Jilin, China in April--May, 2026. He is grateful to the Academy for its hospitality and support.

\section*{Funding} The research of L.T.T. Huyen and N.V. Tuyen is funded by Hanoi Pedagogical University 2 Foundation for Sciences and Technology Development under Grant number HPU2.2025-UT-10. The research of L. Jiao is partially supported by the Chinese National Natural Science Foundation under grant number 12371300. 

\section*{Disclosure Statement} 
The authors declare that they have no conflict of interest.

\section*{Data Availability} There is no data included in this paper.

\bibliographystyle{amsplain}

\end{document}